\documentclass[12pt]{amsart}
\usepackage[latin1]{inputenc}
\usepackage{amsmath} 
\usepackage{amsfonts}
\usepackage{amssymb}
\usepackage{stmaryrd}
\usepackage{latexsym} 
\usepackage{graphicx}
\usepackage{subfigure}
\usepackage{color}
\usepackage{hyperref}
\usepackage{verbatim}
\usepackage[all]{xy}
\usepackage{graphics}
\usepackage{pdfsync}

\theoremstyle{plain}
\newtheorem{theorem}{Theorem}[section]
\newtheorem{lemma}[theorem]{Lemma}
\newtheorem{corollary}[theorem]{Corollary}

\newtheorem{definition}[theorem]{Definition}
\newtheorem{example}[theorem]{Example}

\theoremstyle{remark}
\newtheorem{remark}[theorem]{Remark}

\title[Tableaux and marriage problems]{A generalization of balanced tableaux and marriage problems with unique solutions}

\author{Brian T. Chan}
\address{
 Department of Mathematics \\
 University of British Columbia\\
 Vancouver BC V6T 1Z2, Canada}
\email{bchan@math.ubc.ca}

\date{\today}
\thanks{The author was supported in part by the Natural Sciences and Engineering Research Council of Canada \includegraphics[scale = 0.2]{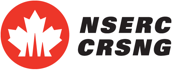} [funding reference number PGSD2 - 519022 - 2018].}

\newlength\cellsize 
\savebox2{%
\begin{picture}(15,15)
\put(0,0){\line(1,0){15}}
\put(0,0){\line(0,1){15}}
\put(15,0){\line(0,1){15}}
\put(0,15){\line(1,0){15}}
\end{picture}}
\newcommand\cellify[1]{\def\thearg{#1}\def\nothing{}%
\ifx\thearg\nothing
\vrule width0pt height\cellsize depth0pt\else
\hbox to 0pt{\usebox2\hss}\fi%
\vbox to 15\unitlength{
\vss
\hbox to 15\unitlength{\hss$#1$\hss}
\vss}}
\newcommand\tableau[1]{\vtop{\let\\=\cr
\setlength\baselineskip{-16000pt}
\setlength\lineskiplimit{16000pt}
\setlength\lineskip{0pt}
\halign{&\cellify{##}\cr#1\crcr}}}
\savebox3{%
\begin{picture}(15,15)
\put(0,0){\line(1,0){15}}
\put(0,0){\line(0,1){15}}
\put(15,0){\line(0,1){15}}
\put(0,15){\line(1,0){15}}
\end{picture}}
\newcommand\expath[1]{%
\hbox to 0pt{\usebox3\hss}%
\vbox to 15\unitlength{
\vss
\hbox to 15\unitlength{\hss$#1$\hss}
\vss}}
\newcommand\bas[1]{\omit \vbox to \cellsize{ \vss \hbox to \cellsize{\hss$#1$\hss} \vss}}

\begin{document}

\begin{abstract} We consider families of finite sets that we call shellable and that have been characterized by Chang and by Hirst and Hughes as being the families of sets that admit unique solutions to Hall's marriage problem. In this paper, we introduce a natural generalization of Edelman and Greene's balanced tableaux that involves families of sets that satisfy Hall's marriage Condition and certain words in $[n]^m$, then prove that shellable families can be characterized by a strong existence condition relating to this generalization. As a consequence of this characterization, we show that the average number of such generalized tableaux is given by a generalization of the hook-length formula.
\end{abstract}
\maketitle
%
\section{Introduction} \label{sec:intro}
\emph{Hall's Marriage Theorem} is a combinatorial theorem that characterises when a finite family of sets has a system of distinct representatives, which is also called a transversal. Hall \cite{HMT} proved that such a family has a system of distinct representatives if and only if this family satisfies the marriage condition. This theorem is known to be equivalent to at least six other theorems \cite{ECMT} which include Dilworth's Theorem, Menger's Theorem, and the Max-Flow Min-Cut Theorem. Hall Jr. proved \cite{HMTE} that Hall's Marriage Theorem also holds for arbitrary families of finite sets. Afterwards, Chang \cite{SDRC} noted how Hall Jr.'s work in \cite{HMTE} can be used to characterize marriage problems with unique solutions. Specifically, the families of sets that admit marriage problems with unique solutions were characterized \cite{SDRC}. Later on, Hirst and Hughes proved that such a characterization of marriage problems with unique solutions can be derived by only using a subsystem of second order arithmetic known as $RCA_0$ \cite{US}, and they showed that their work in \cite{US} can also be extended to marriage problems with a fixed finite number of solutions \cite{FS}. \\

Standard skew tableaux are well-known and intensively studied in algebraic combinatorics, for example \cite{MK, HSSI, NO, EC2}. Moreover, another class of tableaux was introduced by Edelman and Greene in \cite{BT, BTP}, where they defined balanced tableaux on partition shapes. In investigating the number of maximal chains in the weak Bruhat order of the symmetric group, Edelman and Greene proved \cite{BT, BTP} that the number of balanced tableaux of a given partition shape equals the number of standard Young tableaux of that shape. Since then, connections to random sorting networks \cite{RSN}, the Lascoux-Sch\"utzenberger tree \cite{DL}, and a generalization of balanced tableaux pertaining to Schubert polynomials \cite{BLSP} have been explored. Lastly, properties of products of hook-lengths have recently enjoyed some attention by Pak et.al. \cite{AST, HI} and by Swanson \cite{MT}. In particular, an inequality between products of hook-lengths and products of dual hook-lengths was derived \cite{AST, HI, MT}. \\

In this paper, we call the families of finite sets that admit marriage problems with unique solutions \emph{shellable} and give a new characterization of these families of sets by generalizing standard skew tableaux and Edelman and Greene's balanced tableaux to families with systems of distinct representatives, generalizing hook sets to members of such families, and using certain words in $[n]^m$. We then use our characterization of marriage problems with unique solutions to show that the average number of such generalized tableaux is given by a generalization of the hook-length formula. Afterwards, we indicate how our generalization of standard skew tableaux and balanced tableaux can be analysed using Naruse's Formula for skew tableaux and how such an approach can be extended to skew shifted shapes \cite{GK, AST, NO} and likely to certain $d$-complete posets \cite{GK, NO}. \\

\section{Preliminaries} \label{sec:prelims}

Throughout this paper, let $\mathbb{N}$ denote the set of positive integers and for all $n \in \mathbb{N}$, define $[n] = \{1,2,\dots,n\}$. Let $X$ and $Y$ be sets, and let $f : X \to Y$ be a function. Write $X \backslash Y = \{x \in X : x \notin Y\}$. Moreover, for all $X' \subseteq X$, let $f(X') = \{f(r) : r \in X'\}$, and, for all $Y' \subseteq Y$, let $f^{-1}(Y') = \{r \in Y : f(r) \in Y'\}$. If $X = Y$ and $f$ is a bijection, then let $f^{-1}$ denote the inverse of $f$. For all $X' \subseteq X$, let the \emph{restriction of $f$ to $X'$}, which we denote by $f|_{X'}$, be the function $g : X' \to Y$ defined by $g(r) = f(r)$ for all $r \in X'$. For all $m,n \in \mathbb{N}$, say that a function $f : [n] \to [m]$ is \emph{order-preserving} if for all $1 \leq i \leq j \leq n$, $f(i) \leq f(j)$. Lastly, we write examples of permutations using one-line notation. When describing families of sets, we write $\mathcal{F} = \{F_i : i \in I\}$ for some indexing set $I$, write $F \in \mathcal{F}$ to mean that $F = F_i$ for some $i \in I$, and call $F \in \mathcal{F}$ a \emph{member} of $\mathcal{F}$. We treat families of sets as multisets, so the members of $\mathcal{F}$ are counted with multiplicity (that is, $|\mathcal{F}| = |I|$ if $\mathcal{F} = \{F_i : i \in I\}$). Moreover, if $\mathcal{F}$ is a family of sets and if $X$ is a set, then a function, or map, $t : \mathcal{F} \to X$ is an assignment that assigns every member $F$ of $\mathcal{F}$ to exactly one element, denoted $t(F)$, of $X$. \\

An illustrative class of examples that we use in this paper will come from skew shapes. Hence, we recall them below and describe the notation we will use. A \emph{partition $\lambda$} is a weakly decreasing sequence of positive integers. We write $\lambda = (\lambda_1, \lambda_2, \dots, \lambda_\ell)$ to denote such a partition, where $\lambda_i \in \mathbb{N}$ for all $1 \leq i \leq \ell$. If $\lambda$ is a partition, then we will also represent it as a \emph{Young diagram}, which we also denote by $\lambda$. Specifically, the \emph{Young diagram of $\lambda = (\lambda_1, \lambda_2, \dots, \lambda_\ell)$} is a left-justified array of boxes with $\ell$ rows and with $\lambda_i$ boxes in the $i^{th}$ top-most row. If $\lambda = (\lambda_1, \lambda_2, \dots, \lambda_\ell)$ and $\mu = (\mu_1, \mu_2, \dots, \mu_{\ell'})$ are partitions such that $\ell' \leq \ell$ and $\mu_i \leq \lambda_i$ for all $1 \leq i \leq \ell'$, then define the \emph{skew shape} $\lambda / \mu$ to be the shape obtained from the Young diagram of $\lambda$ by deleting the $\mu_1$ left-most cells of row $1$ of $\lambda$, deleting the $\mu_2$ left-most cells of row $2$ of $\lambda$, etc. We also consider a Young diagram $\lambda$ as the skew shape $\lambda / \mu$ where $\mu$ is a partition called the \emph{empty partition}. Lastly, given a skew shape $\lambda / \mu$, we write $(i,j)$ to denote the cell of $\lambda / \mu$ located in the $i^{th}$ top-most row of $\lambda / \mu$ and the $j^{th}$ left-most column of $\lambda / \mu$.

\section{Shellable families of sets and words in $[n]^m$}

In this section, we consider families of sets that satisfy Hall's marriage condition. By interpreting the members of such families as generalized hook-sets and generalizing standard skew tableaux and Edelman and Greene's balanced tableaux to this setting, we give a new characterization of marriage problems with unique solutions. The characterization involves certain words on $[n]^m$ that are defined on such families of sets. We then show that the average number of such generalized tableaux is given by a generalization of the hook-length formula. Lastly, we indicate a connection between our results and known results relating to Naruse's Formula for standard skew tableaux. \\

A well-known notion for families of sets is the following.

\begin{definition} \label{transversal} Let $n \in \mathbb{N}$, and let $\mathcal{F}$ be a finite family of subsets of $[n]$. Then a \emph{transversal} of $\mathcal{F}$ is an injective function $t : \mathcal{F} \to [n]$ such that $t(F) \in F$ for all $F \in \mathcal{F}$. The set $\{t(F) : F \in \mathcal{F} \}$ is called a \emph{system of distinct representatives} of $\mathcal{F}$.
\end{definition}

\begin{example}\label{transversal example} If $\mathcal{F} = \{F_1, F_2\}$ is the family of sets on $\{1, 2\}$ where $F_1 = \{1,2\}$ and $F_2 = \{1,2\}$, then the injective function $t : \mathcal{F} \to \{1, 2\}$ defined by $t(F_1) = 1$ and $t(F_2) = 2$ is a transversal of $\mathcal{F}$. However, the function $t : \mathcal{F} \to \{1, 2\}$ defined by $t(F_1) = 1$ and $t(F_2) = 1$ is not a transversal. Furthermore, the family of sets $\{F_1, F_2, F_3\}$ on $\{1,2\}$ where $F_1 = \{1,2\}$, $F_2 = \{1,2\}$, and $F_3 = \{1,2\}$ does not have a transversal, and the power-set of $\{1, 2\}$ does not have a transversal.
\end{example}

Families of sets that have transversals are of great interest. Exemplary of this is \emph{Hall's Marriage Theorem}, which we present below.

\begin{definition}(Hall, \cite{HMT}) \label{marriage condition} Let $n \in \mathbb{N}$, and let $\mathcal{F}$ be a finite family of subsets of $[n]$. Then $\mathcal{F}$ satisfies the \emph{marriage condition} if for all subfamilies $\mathcal{F}'$ of $\mathcal{F}$,
$$ |\mathcal{F}'| \, \leq \, \bigg{|} \bigcup_{F \in \mathcal{F}'} F \, \bigg{|}. $$
\end{definition}

\begin{theorem}(Hall, \cite{HMT}) \label{marriage theorem} Let $n \in \mathbb{N}$, and let $\mathcal{F}$ be a family of non-empty subsets of $[n]$. Then $\mathcal{F}$ has a transversal if and only if $\mathcal{F}$ satisfies the marriage condition.
\end{theorem}

In order to meaningfully use the families of sets in Hall's Marriage Theorem, we will define more structure on them.

\begin{definition}\label{structures} Let $n \in \mathbb{N}$, let $\mathcal{F}$ be a family of non-empty subsets of $[n]$, and let $t$ be a transversal of $\mathcal{F}$. Then a \emph{configuration} $f$ of $t$ is a function $f : [n] \to \mathbb{N}$ such that for all $F \in \mathcal{F}$,
$$ f(t(F)) \leq |F|. $$
Moreover, if $m \in [n]$, then say that a surjective map $\sigma : [n] \to [m]$ \emph{satisfies} $f$ if the following holds for all $F \in \mathcal{F}$. The positive integer $\sigma(t(F))$ is the $k^{th}$ smallest element of the set $\sigma(F) = \{\sigma(i) : i \in F\}$, where $k = f(t(F))$.
\end{definition}

\begin{remark} The above definition describes the words in $[n]^m$ that we will be considering in this paper. Specifically, the words in $[n]^m$ that we are interested in are the surjective maps $\sigma : [n] \to [m]$ where $m \leq n$.
\end{remark}

\begin{example} \label{structures ex} Let $\mathcal{F} = \{F_1, F_2\}$ be a family of sets on $\{1, 2\}$ where $F_1 = \{1,2\}$ and $F_2 = \{1,2\}$. As mentioned in Example \ref{transversal example}, the injective function $t : \mathcal{F} \to \{1, 2\}$ defined by $t(F_1) = 1$ and $t(F_2) = 2$ is a transversal of $\mathcal{F}$. Consider three configurations $f'$, $f''$, and $f'''$ of $t$ defined by $f'(1) = 1$ and $f'(2) = 1$, $f''(1) = 1$ and $f''(2) = 2$, and $f'''(1) = 2$ and $f'''(2) = 2$. \\

Regarding $f'$, the following can be said. The surjective map $\sigma : \{1, 2\} \to \{1\}$ satisfies $f'$ because $\sigma(1) = 1$ is the smallest element of $\sigma(F_1) = \sigma(\{1,2\}) = \{1\}$ and because $\sigma(2) = 1$ is the smallest element of $\sigma(F_2) = \sigma(\{1,2\}) = \{1\}$. However, no permutation $\sigma : \{1,2\} \to \{1,2\}$ satisfies $f'$ because supposing otherwise leads to a contradiction: Since $\sigma$ satisfies $f$, $f'(1) = 1$, and $F_1 = \{1,2\}$, it follows that $\sigma(1)$ is the smallest element of the two-element set $\sigma(\{1,2\}) = \{\sigma(1), \sigma(2)\}$, implying that $\sigma(1) < \sigma(2)$. Similarly, since $\sigma$ satisfies $f$, $f'(2) = 1$, and $F_2 = \{1,2\}$, it follows that $\sigma(2) < \sigma(1)$, a contradiction. \\

Regarding $f''$, the following can be said. The surjective map $\sigma : \{1,2\} \to \{1\}$ does not satisfy $f''$ because $f''(2) = 2$ implies that $\sigma(2)$ is the second smallest element of $\sigma(\{1,2\})$. But that is impossible because $\sigma(\{1,2\}) = \{1\}$. Similarly, the permutation $\sigma = 21$ does not satisfy $f''$ since $f''(2) = 2$ implies that $\sigma(2)$  is the second smallest element of $\sigma(F_2) = \sigma(\{1,2\}) = \{\sigma(1), \sigma(2)\}$, but $\sigma(2) = 1 < 2 = \sigma(1)$. However, the permutation $\sigma = 12$ does satisfy $f''$ because $\sigma(1) = 1$ is the smallest element of $\sigma(F_1) = \{1,2\}$ and because $\sigma(2) = 2$ is the second smallest element of $\sigma(F_2) = \{1,2\}$. \\

Regarding $f'''$, the following can be said. For all integers $m$ satisfying $1 \leq m \leq 2$ and for all surjective maps $\sigma : \{1,2\} \to [m]$, $\sigma$ does not satisfy $f''$. Suppose otherwise. Since $f'''(1) = 2$ and since $\sigma$ satisfies $f'''$, $\sigma(1)$ is the second smallest element of $\sigma(F_1) = \sigma(\{1,2\})$. Moreover, since $f'''(2) = 2$ and since $\sigma$ satisfies $f'''$, $\sigma(2)$ is the second smallest element of $\sigma(F_2) = \sigma(\{1,2\})$. But then, $|\sigma(\{1,2\})| = 2$ and every element of $\sigma(\{1,2\})$ is the largest element of $\sigma(\{1,2\})$, a contradiction.
\end{example}

In explaining examples for skew shapes, we will need the following definitions.

\begin{definition}\label{tableaux} Let $\lambda / \mu$ be a skew shape with $n$ cells, and let $1 \leq m \leq n$ be a positive integer. Then a \emph{generalized semi-standard skew tableau of shape} $\lambda / \mu$ is a filling $T$ of the cells of $\lambda / \mu$ with numbers from $[m]$ such that the set of entries in $T$ is $[m]$. Moreover, a \emph{generalized standard skew tableau of shape} $\lambda / \mu$ is a bijective filling of the cells of $\lambda / \mu$ with numbers from $[n]$. Moreover, a \emph{standard skew tableau of shape} $\lambda / \mu$ is a generalized standard skew tableau of shape $\lambda / \mu$ such that the entries along every row increase from left to right and the entries along every column increase from top to bottom.
\end{definition}

\begin{remark} For our purposes, we will regard generalized semi-standard skew tableaux as special cases of the surjective maps $\sigma : [n] \to [m]$ from Definition \ref{structures} and we will regard generalized standard skew tableaux are special cases of such maps $\sigma : [n] \to [m]$ when $m = n$. In particular, as we will explain later in this section, standard skew tableaux can be regarded as examples of permutations $\sigma : [n] \to [n]$ that satisfy a fixed configuration $f$.
\end{remark}

\begin{example} If $\lambda = (4,3,1)$ and $\mu = (2)$, then an example of a generalized standard skew tableau of shape $\lambda / \mu$ is
\begin{center}
$\tableau{&& 3 & 5 \\ 6 & 1 & 2 \\ 4}$
\end{center}
and an example of a standard skew tableau of shape $\lambda / \mu$ is
\begin{center}
$\tableau{&& 2 & 3 \\ 1 & 5 & 6 \\ 4}$ \; .
\end{center}
Lastly, an example of a generalized semi-standard skew tableau $T$ of shape $\lambda / \mu$ is
\begin{center}
$\tableau{&& 2 & 3 \\ 3 & 1 & 2 \\ 2}$
\end{center}
since the set of entries in $T$ is $\{1,2,3\}$.
\end{example}

In order to fully relate Definition \ref{tableaux} to Definition \ref{structures}, we will use the following standard definitions.

\begin{definition}\label{standard hooks} Let $\lambda / \mu$ be a skew shape. For any cell $(i, j) \in \lambda$, define the corresponding \emph{hook} $H_{(i, j)}$ to be:
$$H_{(i, j)} = \{(i', j') \in \lambda : i' \geq i \text{ and } j' \geq j \}$$
and the corresponding \emph{hook-length} $h_{(i,j)}$ to be:
$$h_{(i,j)} = |H_{(i,j)}|. $$
We also write $H_{(i,j)} = H_r$ and $h_{(i,j)} = h_r$ if $r = (i,j)$.
\end{definition}

\begin{example}\label{hook example} Consider the following skew shape $\lambda / \mu$, where $\lambda = (5, 4, 3, 3)$ and $\mu = (2,2,1)$. Moreover, let $r = (2,3)$ be the cell of $\lambda / \mu$ depicted below that is filled with a bullet. Then $H_r$ consists of the cells that are filled with asterisks or bullets, and $h_r = 4$.
\begin{center}
$\tableau{ && {} & {} & {} \\ && \bullet & * \\ & {} & * \\ {} & {} & *}$
\end{center}
\end{example}

Moreover, we define the following which will allow us to more clearly explain the special case of Definition \ref{structures} for skew shapes.

\begin{definition}\label{config} Let $\lambda / \mu$ be a skew shape. A \emph{configuration} of $\lambda / \mu$ is a function $f : \lambda / \mu \to \mathbb{N}$ from the cells of $\lambda / \mu$ to the positive integers so that if $r \in \lambda / \mu$, then $f(r) \in \mathbb{N}$ and $f(r) \leq h_r$. Moreover, let $T$ be a generalized semi-standard skew tableau of shape $\lambda / \mu$ and let $f$ be a configuration of $\lambda / \mu$. Then $T$ \emph{satisfies} $f$ if for all cells $r \in \lambda / \mu$, the entry in cell $r$ of $T$ is the $k^{th}$ smallest, where $k = f(r)$, entry in the set of entries of $T$ that are located in the hook $H_r$.
\end{definition}

\begin{example} Consider the following generalized semi-standard skew tableau of shape $\lambda = (4,3,2)$.
\begin{center}
$\tableau{1 & 2 & 3 & 3 \\ 1 & 2 & 3 \\ 3 & 3 }$
\end{center}
It can be checked that the above tableau satisfies the configuration $f$ of $\lambda$ defined by $f(r) = 1$ for all $r \in \lambda$.
\end{example}

Edelman and Greene introduced the following variant of standard Young tableaux. 

\begin{definition}\label{balanced tableaux}(Edelman and Greene, \cite{BT}) Let $\lambda$ be a normal shape containing $n$ cells. Then a \emph{balanced tableau of shape $\lambda$} is a bijective filling of the cells of $\lambda$ with numbers from $[n]$ such that if $(i,j) \in \lambda$ and if $i'$ is the largest positive integer such that $(i',j) \in \lambda$, if $k = i' - i + 1$, and if $S_{i,j}$ is the set of entries $m$ such that $m$ is contained in a cell in $H_{(i,j)}$, then the entry in cell $(i,j)$ of $\lambda$ is the $k^{th}$ smallest entry of $S_{i,j}$.
\end{definition}

\begin{example} Let $\lambda = (4,3,2)$. Then a balanced tableau of shape $\lambda$ is as follows.
\begin{center}
$\tableau{4 & 5 & 8 & 3 \\ 6 & 7 & 9 \\ 1 & 2 }$
\end{center}
For instance, let $i = 2$ and $j = 1$. Then the entry contained in cell $(i,j)$ of $\lambda$ is $6$. Moreover, the largest integer $i'$ such that $(i',j) \in \lambda$ is $3$, $k = i' - i + 1 = 3 - 2 + 1 = 2$, $H_{(i,j)} = \{(2,1), (2,2), (2,3), (3,1) \}$, and $S_{i,j} = \{1, 6, 7, 9\}$. Lastly, the $k^{th}$ smallest entry of $S_{i,j}$ is $6$, which is the entry in cell $(2,1)$ of the above tableau.
\end{example}

Now, we define the following stronger form of the marriage condition that was defined by Chang \cite{SDRC} and Hirst and Hughes in \cite{US}.

\begin{definition}\label{shellable}(cf. (\cite{US}, Theorem 4)) Let $n \in \mathbb{N}$, let $\mathcal{F}$ be a finite family of subsets of $[n]$, and write $m = |\mathcal{F}|$. Then $\mathcal{F}$ is \emph{shellable} if there exists a bijection $\sigma_{\mathcal{F}} : [m] \to \mathcal{F}$ such that for all $k \in [m]$,
\begin{equation}\label{shell}
\bigg{|} \bigcup_{i=1}^k \sigma_\mathcal{F}(i) \bigg{|} = k.
\end{equation}
\end{definition}

Informally, $\sigma_{\mathcal{F}}$ maps each $k$ to a subset, such that the union of the first $k$ subsets has cardinality $k$.

\begin{remark}\label{exactly one transversal} Shellable families of sets are connected to Theorem \ref{marriage theorem}. Chang (\cite{SDRC}, Theorem 1) noted that a simple consequence of Hall Jr.'s work (\cite{HMTE}, Theorem 2) is that a finite family $\mathcal{F}$ of subsets of $[n]$ has exactly one transversal if and only if $\mathcal{F}$ is shellable. Later on, Hirst and Hughes showed that this can be proved using a subsystem of second order arithmetic called $RCA_0$ \cite{US} and proved an extension of this result involving infinite families of finite sets in the context of reverse mathematics. From the aforementioned characterization of finite families of subsets of $[n]$ that have exactly one transversal, we have, by Theorem \ref{marriage theorem}, that all shellable families satisfy the marriage condition.
\end{remark}

\begin{remark} The term shellable is not used in \cite{SDRC}, \cite{HMTE}, and \cite{US}. However, we use this terminology because Definition \ref{shellable} resembles the definition of a \emph{shellable pure $d$-dimensional simplicial complex} (\cite{CCG}, Appendix A2.4, Definition A2.4.1). The differences between Definition \ref{shellable} and Definition A2.4.1 are as follows. The sets in Definition \ref{shellable} do not require additional conditions on the cardinalities of the members of $\mathcal{F}$. Also, in Definition A2.4.1, the requirement of the existence of a bijection $\sigma_{\mathcal{F}} : [m] \to \mathcal{F}$ as described in Definition \ref{shellable} is relaxed to requiring the existence of a certain bijection from a subset of $[m]$ to a subset of $\mathcal{F}$.
\end{remark}

\begin{remark}\label{total order} When describing the members of a shellable family, we will use total orderings of the form $\leq_{\sigma, i}$, defined below, on the members of that family. Let $\mathcal{F}$ be a shellable family of subsets of $[n]$ and let $m = |\mathcal{F}|$. For every bijection $\sigma : [m] \to \mathcal{F}$ that satisfies Equation \ref{shell} in Definition \ref{shellable} for all $k \in [m]$, define two total orderings $\leq_{\sigma, 0}$ and $\leq_{\sigma, 0}$ on the members of $\mathcal{F}$ by defining, for all members $F', F'' \in \mathcal{F}$, $F' <_{\sigma, 0} F''$ if $\sigma^{-1}(F') < \sigma^{-1}(F'')$ and $F' <_{\sigma,1} F''$ if $\sigma^{-1}(F'') < \sigma^{-1}(F')$. The \emph{shelling order} of a shellable complex from (\cite{CCG}, Appendix A2.4, Definition A2.4.1) is a variant of the total orderings just mentioned. \\
\end{remark}

Many families of sets that satisfy the marriage condition are not shellable. For instance, the family $\mathcal{F} = \{F_1, F_2\}$ from Example \ref{transversal example}, where $F_1 = \{1,2\}$ and $F_2 = \{1,2\}$, satisfies the marriage condition but is not shellable. However, many families of sets that satisfy the marriage condition are shellable. For instance, we explain why the notions we are introducing generalize Edelman and Greene's balanced tableaux and standard skew tableaux. Let $\lambda$ be a Young diagram. Then an \emph{inner corner of $\lambda$} is a cell $r \in \lambda$ such that deleting $r$ from $\lambda$ results in another Young diagram \cite{Sagan}. For instance, if $\lambda = (4,2,2)$, then the inner corners of $\lambda$ are the cells filled with bullets.
\begin{center}
$\tableau{{} & {} & {} & \bullet \\ {} & {} \\ {} & \bullet}$
\end{center}
With this definition in mind, let $\lambda / \mu$ be a skew shape with $n$ cells, and consider the family of sets defined by $\mathcal{F} = \{H_r : r \in \lambda / \mu\}$. Then $\mathcal{F}$ is shellable. To see this, let $r_1, r_2, \dots, r_n$ be a sequence of cells in $\lambda / \mu$ that is obtained as follows. \\
\begin{itemize}
\item Let $r_1$ be an inner corner of $\lambda$. \\
\item If $1 \leq k < n$ and if $r_1$, $r_2$, $\dots$, $r_k$ have already been defined, then let $\lambda^{(k)}$ be the Young diagram that results from deleting cells $r_1$, $r_2$, $\dots$, and $r_k$ from $\lambda$, and let $r_{k+1}$ be an inner corner of $\lambda^{(k)}$. \\
\end{itemize}
Lastly, let $\lambda^{(n)} = \mu$. Define $\sigma_{\mathcal{F}} : [n] \to \mathcal{F}$ by letting $ \sigma_{\mathcal{F}}(k) = H_{r_k}$ for all $k \in [n]$. It can be checked that the bijection $\sigma_{\mathcal{F}}$ satisfies Equation \ref{shell}. Hence,  $\mathcal{F}$ is shellable by Definition \ref{shellable}. In particular, by Remark \ref{exactly one transversal}, $\mathcal{F}$ has a unique transversal. The unique transversal $t : \mathcal{F} \to \lambda / \mu$ of $\mathcal{F}$ is given by $t(H_r) = r$ for all $r \in \lambda / \mu$. \\

Let $\lambda / \mu$ be a skew shape with $n$ cells, consider the shellable family of sets $\mathcal{F} = \{H_r : r \in \lambda / \mu \}$, and let $t$ be the unique transversal of $\mathcal{F}$. Define the configuration $f_0$ of $t$ by $f_0(r) = 1$ for all $r \in \lambda / \mu$. It can be seen that the standard Young tableaux of shape $\lambda / \mu$ are the generalized standard skew tableaux of shape $\lambda / \mu$ that satisfy $f_0$. In particular, the standard skew tableau of shape $\lambda / \mu$ can be regarded as being the permutations $\sigma : [n] \to [n]$ that satisfy $f_0$. \\

Next, let $\lambda$ be a partition of $n$, consider the shellable family of sets $\mathcal{F} = \{H_r : r \in \lambda \}$, and let $t$ be the unique transversal of $\mathcal{F}$. Define the configuration $f$ of $t$ such that, for all $(i,j) \in \lambda$, if $i'$ is the largest positive integer such that $(i',j) \in \lambda$, then $f((i,j)) = i' - i + 1$. Then a balanced tableaux of shape $\lambda$ is a generalized skew tableau of shape $\lambda / \mu$ that satisfies $f$. In particular, balanced tableaux can be regarded as permutations $\sigma : [n] \to [n]$ that satisfy $f$. The aforementioned reformulation of the definition of balanced tableaux involving generalized skew tableaux was used in \cite{BT} as the definition of balanced tableaux; the special case of Definition \ref{config} for normal shapes also appears in \cite{BT} under a different name. Namely, Edelman and Greene call $f(r)$ the \emph{hook rank} of $r$. However, they only use hook ranks to define balanced tableaux. In this paper, we have a very different emphasis as we focus on properties of the configurations themselves. \\

Before proving the main result of this paper, we prove the following lemma. \\

\begin{lemma}\label{not empty} Let $n \in \mathbb{N}$ and assume that $\mathcal{F}$ is a shellable family of subsets of $[n]$ such that $|\mathcal{F}| = n$. Moreover, let $S \subseteq [n]$ be the set of elements $k \in [n]$ such that $k \in F$ for exactly one member $F$ of $\mathcal{F}$. Then $S$ is not empty.
\end{lemma}
\begin{proof} If $n = 1$, then $\mathcal{F} = \{\{1\}\}$, implying that $S = \{1\}$. So assume without loss of generality that $n \geq 2$. Because $\mathcal{F}$ is shellable, there is an ordering $\mathcal{F} = \{F_1, F_2, \dots, F_n \}$ of the members of $\mathcal{F}$ that satisfies Equation \ref{shell} of Definition \ref{shellable}. In particular, $|\bigcup_{i = 1}^{n-1} F_i| = n - 1$ and $|\bigcup_{i=1}^n F_i| = n$. Let $k \in \bigcup_{i=1}^n F_i \; \backslash \; \bigcup_{i=1}^{n-1} F_i$. Then $k \in F_n$ but,  for all $1 \leq i \leq n-1$, $k \notin F_i$. Hence, $k$ is contained in exactly one member of $\mathcal{F}$ and so $k \in S$. It follows that $S$ is not empty.
\end{proof}

\begin{remark} A family $\mathcal{F}$ of subsets of $[n]$ such that $|\bigcup_{F \in \mathcal{F}} F| = |\mathcal{F}| = n$ is called a \emph{critical block} in \cite{HMTE}. In \cite{HMTE}, Hall Jr. used this notion as a very important ingredient in extending Hall's Marriage Theorem to infinite families of finite sets.
\end{remark}

Now, we prove the main result of this paper. \\

\begin{theorem}\label{good marriage} Let $n \in \mathbb{N}$, let $\mathcal{F}$ be a family of subsets of $[n]$ such that $|\mathcal{F}| = n$, assume that $\mathcal{F}$ satisfies the marriage condition, and let $t$ be a transversal of $\mathcal{F}$. Moreover, let $S \subseteq [n]$ be the set of elements $k \in [n]$ such that $k \in F$ for exactly one member $F$ of $\mathcal{F}$. Lastly, let $m$ be an integer satisfying
$$\min(n, n - |S| + 1) \leq m \leq n.$$
Then $\mathcal{F}$ is shellable if and only if for all configurations $f$ of $t$, there exists a surjective map $\sigma : [n] \to [m]$ such that $\sigma$ satisfies $f$.
\end{theorem}

\begin{example} We give an example in which the lower bound $\min(n, n - |S| + 1)$ cannot be improved on. Consider $\lambda = (3,2,1)$. Next, let $\mathcal{F} = \{H_r : r \in \lambda \}$ and let $t$ be the transversal of $\mathcal{F}$ defined by $t(H_r) = r$ for all $r \in \lambda$. As discussed earlier, $\mathcal{F}$ is shellable. Now, let $f$ be the configuration of $\lambda$ be defined by $f((1,1)) = 5$, $f((1,2)) = 3$, $f((1,3)) = 1$, $f((2,1)) = 3$, $f((2,2)) = 1$, and $f((3,1)) = 1$. This configuration $f$ is depicted below, where we fill every cell $r \in \lambda$ with $f(r)$.
\begin{center}
$\tableau{5 & 3 & 1 \\ 3 & 1 \\ 1 }$
\end{center}

There is exactly one cell in the Young diagram $\lambda$, the cell $(1,1)$, that is contained in exactly one member of $\mathcal{F} =  \{H_r : r \in \lambda \}$. Hence, $S = \{(1,1)\}$ and $\min(n, n - |S| + 1) = n - |S| + 1$. With this in mind, set $n = 6$ and, as $n - |S| + 1 = 6 - 1 + 1 = 6$, assume that $m$ is an integer satisfying $1 \leq m \leq 5$. Suppose that there exists a generalized skew tableau $T$ of shape $\lambda$, and with entries from $[m]$, such that $T$ satisfies the configuration $f$ defined above. The cells $(1,1)$, $(1,2)$ and $(2,1)$ are cells $r \in \lambda$ that satisfy $f(r) = h_r$. Moreover, because $T$ satisfies $f$, $f(r) = h_r$ implies that no two entries of $T$ in $H_r$ are the same and that the entry of $T$ in cell $r$ is the $h_r^{th}$ smallest element of the set of entries of $T$ that are contained in $H_r$. \\

So consider the cell $(2,2)$ of $\lambda$. Since $m \leq 5$, some two entries of $T$ in $H_{(1,1)}$ are the same, or the entry of $T$ in cell $(2,2)$ equals to the entry of $T$ in some other cell, $(i_1,j_1)$, in $\lambda$. Since $f((1,1)) = 5 = h_{(1,1)}$, no two entries of $T$ in $H_{(1,1)}$ are the same. So the entry of $T$ in cell $(2,2)$ equals to the entry of $T$ in some other cell, $(i_1,j_1)$, in $\lambda$. If $(i_1, j_1) = (1,1)$, then the entry of $T$ in cell $(2,2)$ of $\lambda$ is larger than the entries of $T$ in cells $(1,2)$ and $(2,1)$ of $\lambda$. But that is impossible by the above. If $(i_1, j_1) = (2,1)$ or if $(i_1, j_1) = (3,1)$, then the entry of $T$ in cell $(2,1)$ of $\lambda$ is the $k^{th}$ smallest element of the set of entries of $T$ that are contained in $H_{(2,1)}$ for some $k \leq 2$. But that is impossible by the above. By symmetry, it is impossible for $(i_1, j_1) = (2,2)$ or for $(i_1, j_1) = (1,3)$. Hence, we have reached a contradiction. It follows that there is no such generalized skew tableau $T$.
\end{example}

\begin{proof} Let $n$, $\mathcal{F}$, $t$, $S$, and $m$ be as described in the theorem. First assume that for all configurations $f$ of $t$, there exists a surjective map $\sigma : [n] \to [m]$ that satisfies $f$. If $n = 1$, then the only family of $\{1\}$ with a transversal is the family $\mathcal{F} = \{ \{1\} \}$, which is shellable. \\

So assume without loss of generality that $n \geq 2$. Consider the configuration $f_1$ of $t$ defined by $f_1(t(F)) = |F|$ for all $F \in \mathcal{F}$.  By assumption, there exists a surjective map $\sigma' : [n] \to [m]$ that satisfies $f_1$. Moreover, let $k \in [n-1]$, and assume that we can fix an ordering $\mathcal{F} = \{F'_i : i \in [n] \}$ of $\mathcal{F}$ so that the following holds for all integers $0 \leq j \leq k - 1$.
\begin{equation}\label{shellability condition from max index}
\bigg{|}\bigcup_{i=1}^{n-j} F'_i\bigg{|} = n - j
\end{equation}
Note that Equation \ref{shellability condition from max index} holds if $k = 1$ because the fact that $\mathcal{F}$ has a transversal implies that $\bigcup_{F \in \mathcal{F}} F = [n]$. \\

Next, let $1 \leq s \leq n - k + 1$ satisfy
\begin{equation}\label{maximal index}
\sigma'(t(F'_s)) \; = \max_{1 \leq j \leq n - k + 1} \sigma'(t(F'_j)).
\end{equation}
Suppose that there exists an element $j \in [n]$ such that $1 \leq j \leq n - k + 1$, $j \neq s$, and $t(F'_s) \in F'_j$. By Equation \ref{maximal index}, $\sigma'(t(F'_j)) \leq \sigma'(t(F'_s))$. So as $t(F'_s) \in F'_j$ and $t(F'_s) \neq t(F'_j)$,  it follows that for some $1 \leq \ell \leq |F'_j| - 1$, $\sigma'(t(F'_j))$ is an $\ell^{th}$ smallest element of $\sigma'(F'_j)$. But then, as $f_1(t(F'_j)) = |F'_j|$, $\sigma'$ does not satisfy $f_1$, contradicting the assumption that $\sigma'$ satisfies $f_1$. \\

Hence, $t(F_s') \notin F_i'$ for all $1 \leq i \leq n - k + 1$ satisfying $i \neq s$. In particular, fix an ordering $\mathcal{F} = \{F''_i : i \in [n] \}$ of the members of $\mathcal{F}$ so that $F''_i = F'_i$ if $i > n - k + 1$ and $F''_{n-k+1} = F'_s$, where $s$ is as described in the above paragraph. By Equation \ref{shellability condition from max index} and the fact that $t(F_s') \notin F_i'$ for all $1 \leq i \leq n - k + 1$ satisfying $i \neq s$, it follows that this ordering of the members of $\mathcal{F}$ satisfies the following equation for all integers $0 \leq j \leq k$.
$$\bigg{|}\bigcup_{i=1}^{n-j} F''_i\bigg{|} = n - j $$
As the choice of $k \in [n-1]$ is arbitrary, it follows that there exists an ordering $\mathcal{F} = \{F_1, F_2, \dots, F_n\}$ of $\mathcal{F}$ such that
$$ \bigg{|} \bigcup_{i=1}^k F_i \bigg{|} = k $$
for all $1 \leq k \leq n$. Hence, $\mathcal{F}$ satisfies Equation \ref{shell} of Definition \ref{shellable}. So, by Definition \ref{shellable}, $\mathcal{F}$ is shellable. \\

Next, assume that $\mathcal{F}$ is shellable. We proceed by strong induction on $n$. Because $\mathcal{F}$ is shellable, we will use the total orderings as described in Remark \ref{total order} to describe the members of this family. \\

If $n = 1$, then the only family of subsets of $\{1\}$ with a transversal is the family $\mathcal{F} = \{\{1\}\}$. Moreover, with $t$ being the transversal of $\mathcal{F}$ defined by mapping $\{1\}$ to $1$, the only configuration $f$ that satisfies $t$ is the function $f : \{1\} \to \mathbb{N}$ defined by $f(1) = 1$, $S = \{1\}$, $\min(n, n - |S| + 1) = 1$, and the surjective map $\sigma : \{1\} \to \{1\}$ satisfies $f$. \\

So assume that $n \geq 2$ and let $f$ be a configuration of $t$. Since $S$ is not empty by Lemma \ref{not empty}, $\min(n, n - |S| + 1) = n - |S| + 1$, implying that and $n - |S| + 1 \leq m \leq n$. Assume without loss of generality that
\begin{equation}\label{S indexing}
S = \{n - m' + 1, n - m' + 2, \dots, n \}
\end{equation}
for some $m' \in [n]$. If $m = 1$, then $n - |S| + 1 \leq 1$, implying that $n = |S|$. Hence, if $m = 1$, then, as $|\mathcal{F}| = n$ and $S = [n]$, every element of $[n]$ is contained in exactly one element of $\mathcal{F}$. In particular, if $m = 1$, then $\mathcal{F} = \{\{k\} : k \in [n]\}$, $t(\{k\}) = k$ for all $k \in [n]$, the only configuration $f$ of $t$ is the map defined by $f(k) = 1$ for all $k \in [n]$, and the surjective map $\sigma : [n] \to [m]$, defined by $\sigma(k) = 1$ for all $k \in [n]$, satisfies $f$. So assume without loss of generality that $m \geq 2$. \\

Since $n - |S| + 1 \leq m \leq n$, $m$ satisfies the inequality $n - m '+ 1 \leq m \leq n$. As $\mathcal{F}$ is shellable, there is an ordering $\mathcal{F} = \{F'_1, F'_2, \dots, F'_n\}$ of the members of $\mathcal{F}$ such that 
\begin{equation}\label{another shellable ordering}
\bigg{|} \bigcup_{i=1}^k F'_i \bigg{|} = k
\end{equation}
for all $1 \leq k \leq n$. Define the following subfamilies of $\mathcal{F}$, 
$$ \mathcal{F}_0 = \{F \in \mathcal{F} : t(F) \leq m - 1 \}$$
and
$$ \mathcal{F}_1 = \{F \in \mathcal{F} : t(F) \geq m\}. $$

We first prove that $\mathcal{F}_0$ is shellable. Because $S$ is the set of elements $k \in [n]$ such that $k \in F$ for exactly one member $F$ of $\mathcal{F}$, Equation \ref{S indexing} and the fact that $n - m' + 1 \leq m \leq n$ implies that for all $m \leq k \leq n$, $k$ is contained in exactly one member of $\mathcal{F}$ and that for all $F \in \mathcal{F}_1$, 
$$|F \cap \{m, m+1, \dots, n\}| = 1.$$ 
In particular, $\mathcal{F}_0$ is an $(m - 1)$-member family of subsets of $[m-1]$. \\

Assume that there exists an integer $1 \leq j \leq n-1$ such that $F'_j \in \mathcal{F}_1$ and $F'_{j+1} \in \mathcal{F}_0$. Write 
$$X = \bigcup_{i=1}^{j-1} F'_i,$$
where we assume that $X = \emptyset$ if $j = 1$. Since $F'_j \in \mathcal{F}_1$, $t(F'_j) \in \{m, m+1, \dots, n\}$ and no member of $\mathcal{F}$ other than $F'_j$ contains $t(F'_j)$. Moreover, by Equation \ref{another shellable ordering}, $|F'_j \cup X| = |X| + 1$. So as $t(F'_j) \in F'_j$, it follows that $[m-1] \cap (F'_j \cup X) = [m-1] \cap X$. Since $F'_{j+1} \in \mathcal{F}_0$, $F'_{j+1} \subseteq [m-1]$. Moreover, by Equation \ref{another shellable ordering}, $|X \cup F'_j \cup F'_{j+1}| = |X \cup F'_j| + 1$. It follows that $F'_{j+1} \backslash X = F'_{j+1} \backslash (X \cup F'_j) = \{k\}$ for some $k \in [m-1]$, implying that 
\begin{equation}\label{transposition}
|F'_{j+1} \cup X| = |X| + 1.
\end{equation}
So the ordering $\mathcal{F} = \{F''_1, F''_2, \dots, F''_n\}$ of the members of $\mathcal{F}$, such that $F''_j = F'_{j+1}$, $F''_{j+1} = F'_j$, and $F''_i = F'_i$ for all $i \in [n] \backslash \{j, j+1\}$, satisfies the following by Equation \ref{another shellable ordering} and Equation \ref{transposition}. For all $1 \leq k \leq n$, 
\begin{equation}
\bigg{|} \bigcup_{i=1}^k F''_i \bigg{|} = k.
\end{equation}
Furthermore, $F''_j \in \mathcal{F}_0$ and $F''_{j+1} \in \mathcal{F}_1$. If there exists an integer $1 \leq j' \leq n - 1$ such that $F''_{j'} \in \mathcal{F}_1$ and $F''_{j'+1} \in \mathcal{F}_0$, then argue again as above. Repeating this argument at most a finite number of times, we obtain an ordering $\mathcal{F} = \{F_1, F_2, \dots, F_n\}$ of the members of $\mathcal{F}$ where
\begin{equation}\label{yet another shelling}
\bigg{|} \bigcup_{i=1}^k F_i \bigg{|} = k
\end{equation}
for all $1 \leq k \leq n$, $\mathcal{F}_0 = \{F_k : 1 \leq k \leq m-1 \}$, and $\mathcal{F}_1 = \{F_k : m \leq k \leq n\}$. In particular, Equation \ref{yet another shelling} holds for all $1 \leq k \leq m-1$, implying that $\mathcal{F}_0$ satisfies Equation \ref{shell} of Definition \ref{shellable}. It follows, by Definition \ref{shellable}, that $\mathcal{F}_0$ is a shellable family of subsets of $[m-1]$. \\

So consider the ordering $\mathcal{F} = \{F_1, F_2, \dots, F_n\}$ of the members of $\mathcal{F}$ as above and assume without loss of generality that for all $m \leq i \leq n$, $t(F_i) = i$. Let $t'$ be the transversal of $\mathcal{F}_0$ defined by $t'(F) = t(F)$ for all $F \in \mathcal{F}_0$. Moreover, let $f' = f|_{[m-1]}$, where $f|_{[m-1]}$ denotes the restriction of $f$ to $[m - 1]$. In particular, $f'$ is a configuration of $t'$. \\

Since it is assumed in the theorem that $\min(n, n - |S| + 1) \leq m \leq n$, and since a surjective map $\sigma : [n] \to [m]$ is a permutation if $m = n$, the following holds. Because $\mathcal{F}_0$ is shellable, because $|\mathcal{F}_0| = m-1$, and because $|[m-1]| < n$, the induction hypothesis implies that there exists a permutation $\sigma' : [m-1] \to [m-1]$ such that $\sigma'$ satisfies $f'$. \\

If there exists an integer $m \leq j \leq n$ such that $f(j) = |F_j|$, then there exists a surjective map $\kappa' : [n] \to [m]$ such that $\kappa'(i) = i$ for all $1 \leq i \leq m-1$ and the following two properties hold for all $m \leq i \leq n$. \\
\begin{itemize}
\item If $f(i) = |F_i|$, then $\kappa'(i) = m$. \\
\item If $f(i) < |F_i|$, then $\kappa'(i)$ is equal to the $f(i)^{th}$ smallest element of $\sigma'(F_i \backslash \{i\})$. \\
\end{itemize}
Otherwise, if $f(i) < |F_i|$ for all $m \leq i \leq n$, then there exists a surjective map $\kappa'' : [n] \to [m]$ such that $\kappa''|_{[m-1]}$ is injective and order-preserving and such that the following two properties hold. \\
\begin{itemize}
\item If $m \leq i < n$, then $\kappa''(i)$ is equal to $f(i)^{th}$ smallest element of $\kappa''(\sigma'(F_i \backslash \{i\}))$. \\
\item If $i = n$, then $\kappa''(i) \notin \kappa''([m-1])$ and $\kappa''(i)$ is equal to the $f(i)^{th}$ smallest element of $\kappa''(i) \cup \kappa''(\sigma'(F_i \backslash \{i\}))$.\\
\end{itemize}

So define a surjective map $\kappa : [n] \to [m]$ as follows. If there exists an integer $m \leq j \leq n$ such that $f(j) = |F_j|$, then define $\kappa = \kappa'$. Otherwise,  if $f(i) < |F_i|$ for all $m \leq i \leq n$, define $\kappa = \kappa''$. Now, define the map $\sigma : [n] \to [m]$ by
$$\sigma(i) = \begin{cases}
\kappa(\sigma'(i)) & \text{ if } 1 \leq i \leq m-1 \\
\kappa(i) & \text{ if } m \leq i \leq n.
\end{cases} $$
Because $\sigma' : [m-1] \to [m-1]$ is a bijection, the definition of $\kappa$ implies that $\sigma$ is surjective. Moreover, because $\sigma'$ satisfies $f'$ and because, for all integers $m \leq i \leq n$, $i$ is contained in exactly one member of $\mathcal{F}$ and $F_i \cap \{m, m+1, \dots, n\} = \{i\}$, the definition of $\kappa$ and the definition of $\sigma$ imply that $\sigma$ satisfies $f$. From this, the theorem follows.
\end{proof}

A natural consequence of the above is the following which, combined with Theorem \ref{good marriage}, gives another characterization of shellable families of sets.

\begin{corollary}\label{nice permutations}  Let $\mathcal{F}$ be a family of subsets of $[n]$ such that $|\mathcal{F}| = n$, assume that $\mathcal{F}$ satisfies the marriage condition, and let $t$ be a transversal of $\mathcal{F}$. Moreover, let $S$ be as in Theorem \ref{good marriage}. Lastly, let $f_0$ be the configuration of $t$ defined by $f_0(t(F)) = 1$ for all $F \in \mathcal{F}$. Then $f_0$ is satisfied by some permutation $\sigma : [n] \to [n]$ if and only if for all integers $n - |S| + 1 \leq m \leq n$ and for all configurations $f$ of $t$, there exists a surjective map $\sigma : [n] \to [m]$ that satisfies $f$.
\end{corollary}

\begin{example} Consider a skew shape $\lambda / \mu$ with $n$ cells, and let $S$ denote the set of cells of $\lambda / \mu$ that are contained in exactly one member of $\{H_r : r \in \lambda / \mu\}$. The elements of $S$ are also known as the \emph{outer corners} of $\mu$ \cite{Sagan}. Clearly, there exists a standard skew tableau of shape $\lambda / \mu$. Corollary \ref{nice permutations} implies that such a tableau exists if and only if for all integers $n - |S| + 1 \leq m \leq n$ and for all configurations $f$ of $\lambda / \mu$, there exists a generalized semi-standard skew tableau $T$ of shape $\lambda / \mu$, with $[m]$ as the set of entries of $T$, such that $T$ satisfies $f$.
\end{example}

\begin{proof} Let $f_1$ be the configuration of $t$ defined by $f_1(t(F)) = |F|$ for all $F \in \mathcal{F}$. Then a permutation $\sigma : [n] \to [n]$ satisfies $f_0$ if and only if the permutation $\sigma' : [n] \to [n]$ defined by $\sigma'(i) = n - \sigma(i) + 1$ for all $i \in [n]$ satisfies $f_1$. In particular, $f_0$ is satisfied by some permutation if and only if $f_1$ is. The first half of the proof of Theorem \ref{good marriage} implies that if $f_1$ is satisfied by some permutation $\sigma : [n] \to [n]$, then $\mathcal{F}$ is shellable. Furthermore, by Theorem \ref{good marriage}, if $\mathcal{F}$ is shellable, then for all integers $n - |S| + 1 \leq m \leq n$ and for all configurations $f$ of $t$, there exists a surjective map $\sigma : [n] \to [m]$ that satisfies $f$. From this, the corollary follows.
\end{proof}

A well-known formula is the \emph{hook-length formula}, first proved by Frame, Robinson, and Thrall \cite{FRT}. It is as follows. If $\lambda$ is a normal shape with $n$ cells, then the number of standard Young tableaux of shape $\lambda$ equals
$$\dfrac{n!}{\prod_{r \in \lambda} h_r}.$$
Moreover, the above formula was also proved by Edelman and Greene to equal the number of balanced tableaux of shape $\lambda$ \cite{BT}. In our context, we will show that the above formula is connected to the configurations that we are investigating. \\

Let $S(n,m)$ denote the Stirling number of the second kind. Namely, let $S(n,m)$ denote the number of set partitions of $[n]$ into $m$ parts. Let $\mathcal{F}$ be a family of sets of $[n]$ that satisfies the marriage condition, let $m \in n$, and let $t$ be a transversal of $\mathcal{F}$. If $f$ is a configuration of $t$, then let $A_{n,m}(f)$ denote the number of surjective maps $\sigma : [n] \to [m]$ that satisfy $f$. Moreover, let $X$ is the set of configurations $f$ of $t$ such that $A_{n,m}(f) \geq 1$. Then define the \emph{average value of $A_{n,m}(f)$ over all configurations $f$ of $t$ satisfying $A_{n,m}(f) \geq 1$} to be
$$\frac{1}{|X|} \sum_{f \in X} A_{n,m}(f) $$
if $|X| > 0$, and $0$ otherwise.

\begin{theorem}\label{theorem : partial product formula} Let $\mathcal{F}$ be a shellable family of subsets of $[n]$ such that $|\mathcal{F}| = n$ and let $t$ be the transversal of $\mathcal{F}$. Moreover, let $S \subseteq [n]$ be the set of elements $k \in [n]$ such that $k \in F$ for exactly one member $F$ of $\mathcal{F}$, and let $m$ be an integer satisfying 
$$ n - |S| + 1 \leq m \leq n.$$
Then the average value of $A_{n,m}(f)$ over all configurations $f$ of $t$ satisfying
$$A_{n,m}(f) \geq 1$$
is
\begin{equation}\label{equation : partial product formula}
\dfrac{m! \, S(n,m)}{\prod_{F \in \mathcal{F}} |F|}. 
\end{equation}
\end{theorem}

\begin{remark} Theorem \ref{theorem : partial product formula} implies the following consequence relating to how the values $A_{n,m}(f)$ are distributed. By Theorem \ref{good marriage}, every configuration $f$ of $t$ is satisfied by at least one surjective map $\sigma : [n] \to [m]$. Hence, by Theorem \ref{theorem : partial product formula} and the fact that $A_{n,m}(f) \geq 1$ always holds, it follows that for all constants $k \geq 1$ the number of configurations $f$ of $t$ that satisfy
$$A_{n,m}(f) \; \leq \; k \cdot \dfrac{m! \; S(n,m)}{\prod_{F \in \mathcal{F}} |F|} $$
is at least
$$\bigg{(}1 - \frac{1}{k}\bigg{)} \prod_{F \in \mathcal{F}} |F|.$$
\end{remark}

\begin{remark}\label{Stirling properties} Expression \ref{equation : partial product formula} is relatively easy to calculate. It is well-known that the Stirling numbers of the second kind satisfy the following explicit formula and the following recurrence relation for all integers $1 \leq m \leq n$ \cite{GuoQi}:
\begin{equation}\label{Stirling explicit formula}
S(n,m) = \dfrac{1}{m!} \sum_{i=0}^m (-1)^i {m \choose i} (m - i)^n
\end{equation}
\begin{equation}\label{Stirling recurrence}
S(n,m) = m \, S(n-1,m) + S(n-1,m-1).
\end{equation}
By Equation \ref{Stirling explicit formula}, it follows that Expression \ref{equation : partial product formula} may be re-written as
\begin{equation}\label{alternative expression}
 \dfrac{1}{\prod_{F \in \mathcal{F}} |F|} \; \sum_{i=0}^m (-1)^i {m \choose i} (m - i)^n.
\end{equation}
In particular, Expression \ref{equation : partial product formula} can be computed with $\mathcal{O}(n^2)$ arithmetic operations. If the difference $n - m$ is fixed, then $S(n,m)$ can be expressed with a closed-form formula. To see this, consider the sequence $(p_k(x))_{k=0,1,2,\dots}$ of polynomials in $\mathbb{Q}[x]$ such that $p_0(x) = 1$ and, for all $k$,
$$p_{k+1}(x) - p_{k+1}(x-1) = x \, p_k(x). $$
If $k = n-m$, then by the recurrence given in Equation \ref{Stirling recurrence}, $S(n,m) = p_k(m)$. So as $k$ is assumed to be fixed, $p_k(x)$ gives a closed-form expression for $S(n,m)$. For instance, Expression \ref{equation : partial product formula} becomes
$$ \dfrac{m!}{\prod_{F \in \mathcal{F}} |F|} $$
if $n = m$,
$$ {m + 1 \choose 2} \; \dfrac{m!}{\prod_{F \in \mathcal{F}} |F|}. $$
if $n = m + 1$, and
$$  \dfrac{1}{2}\,{m + 1 \choose 2} \bigg{(}{m + 1 \choose 2} + \frac{2m + 1}{3} \bigg{)} \; \dfrac{m!}{\prod_{F \in \mathcal{F}} |F|}. $$
if $n = m + 2$.
\end{remark}

In order to prove Theorem \ref{equation : partial product formula}, we prove the following.

\begin{lemma}\label{generalized converse} Let $m, n \in \mathbb{N}$ such that $m \leq n$, and let $\mathcal{F}$ be a family of subsets of $[n]$ that has a transversal $t : \mathcal{F} \to [n]$ such that $t$ is surjective. Then every surjective function $\sigma : [n] \to [m]$ satisfies exactly one configuration $f$ of $t$.
\end{lemma}

\begin{proof} Let $\sigma : [n] \to [m]$ be a surjective map. Then $\sigma$ satisfies the configuration $f$ of $t$ defined by letting, for all $F \in \mathcal{F}$, $f(t(F)) = k$ where $\sigma(t(F))$ is the $k^{th}$ smallest element of the set $\sigma(F)$. Now, suppose that $\sigma$ satisfies more than one configuration of $t$. Then, let $f_1$ and $f_2$ be two distinct configurations of $t$. Because $f_1 \neq f_2$ and because $t$ is surjective, there is an element $F \in \mathcal{F}$ such that $f_1(t(F)) \neq f_2(t(F))$. So write $k_1 = f_1(t(F))$ and write $k_2 = f_2(t(F))$. Since $\sigma$ satisfies $f_1$, Definition \ref{structures} implies that $\sigma(t(F))$ is the $k_1^{th}$ smallest element of $\sigma(F)$. Moreover, since $\sigma$ satisfies $f_2$, Definition \ref{structures} implies that $\sigma(t(F))$ is the $k_2^{th}$ smallest element of $\sigma(F)$. However, this is impossible because $k_1 = f_1(t(F)) \neq f_2(t(F)) = k_2$.
\end{proof}

Now, we prove Theorem \ref{theorem : partial product formula}.

\begin{proof} By Definition \ref{structures} , the total number of configurations of $\mathcal{F}$ equals to $\prod_{F \in \mathcal{F}} |F|$. Moreover, it is well-known that number of surjective maps from $[n]$ to $[m]$ is given by $m! \, S(n,m)$. By Lemma \ref{generalized converse}, every surjective map is satisfied by exactly one configuration. Moreover, by Theorem \ref{good marriage}, every configuration of $\mathcal{F}$ is satisfied by some surjective map from $[n]$ to $[m]$. From this, the theorem follows.
\end{proof}

\begin{example}  Let $\lambda = (6,5,4,3,2,1)$ and $\mu = (1)$. The skew shape $\lambda / \mu$ is depicted below.
\begin{center}
$\tableau{ & & {} & {} & {} & {} \\ & {} & {} & {} & {} \\ {} & {} & {} & {} \\ {} & {} & {} \\ {} & {} \\ {} }$
\end{center}
Since $\lambda / \mu$ has eighteen cells, let $n = 18$. The cells of $\lambda / \mu$ that are contained in exactly one member of the family $\mathcal{F} = \{H_r : r \in \lambda \}$ are $(1,3)$, $(2,2)$, and $(3,1)$. Hence, $S = \{(1,3), (2,2), (3,1) \}$ and $n - |S| + 1 = n - 2$. So let $m = n - 2 = 16$. Then by Theorem \ref{theorem : partial product formula} and Remark \ref{Stirling properties}, the average value of $A_{n,m}(f)$ over all configurations $f$ of $\lambda / \mu$ that satisfy $A_{n,m}(f) \geq 1$ is given by
$$  \dfrac{1}{2}\,{m + 1 \choose 2} \bigg{(}{m + 1 \choose 2} + \frac{2m + 1}{3} \bigg{)} \; \dfrac{m!}{\prod_{F \in \mathcal{F}} |F|} $$
$$ =  \dfrac{1}{2}\,{16 + 1 \choose 2} \bigg{(}{16 + 1 \choose 2} + \frac{2 \cdot 16 + 1}{3} \bigg{)} \; \dfrac{16!}{\prod_{F \in \mathcal{F}} |F|} $$
$$ =  \dfrac{1}{2}\,{17 \choose 2} \bigg{(}{17 \choose 2} + 11 \bigg{)} \; \dfrac{16!}{(7 \cdot 5 \cdot 3 \cdot 1)^3 \cdot 5 \cdot 3 \cdot 1 \cdot 3 \cdot 1 \cdot 1} $$
$$ = 4014814003 + \dfrac{1}{5}. $$
\end{example}

\begin{remark}\label{av appl} Theorem \ref{theorem : partial product formula} is versatile. For instance, possible applications of the special case of Theorem \ref{theorem : partial product formula} in the case of permutations are as follows. There is a formula for the number of standard skew tableaux of shape $\lambda / \mu$, known as Naruse's formula. Asymptotic properties of Naruse's formula were analysed by Morales, Pak, and Panova in \cite{AST}. In particular \cite{AST}, it turns out that in general, the number of standard skew tableaux of shape $\lambda / \mu$ divided by
$$ \dfrac{n!}{\prod_{r \in \lambda / \mu} h_r}, $$
where $n$ is the number of cells of $\lambda / \mu$,
can be arbitrarily large. Hence, we can apply Theorem \ref{theorem : partial product formula} to Naruse's formula and, using the work of Morales, Pak, and Panova in \cite{AST}, analyse lower bounds on the number of configurations $f$ of $\lambda / \mu$ such that $A_{n,n}(f) \geq 1$ and $A_{n,n}(f)$ is strictly less than
$$ \dfrac{n!}{\prod_{r \in \lambda / \mu} h_r}. $$
\end{remark}

\begin{remark} Regarding Remark \ref{av appl}, there are variants and generalizations of Naruse's formula, the formula described in Remark \ref{av appl}, for shapes known as \emph{skew shifted shapes} that are known \cite{GK, NO}. What we observe about these shapes is that the ``hook-sets'' for skew shifted shapes as defined in \cite{GK, NO} also form examples of shellable families. Hence, the results in this section can be replicated verbatim to include \emph{skew shifted shapes}. Moreover, it is claimed by Morales, Pak, and Panova in \cite{AST} that their analysis of Naruse's formula can be extended to skew shifted shapes. It also appears that we can even extend the above to involve posets known as \emph{$d$-complete posets} \cite{NO}, as there is a generalization of Naruse's formula for such posets and the ``hook-sets'' in these formulas are a generalization of the ``hook-sets'' for the skew shifted shapes \cite{NO}.
\end{remark} 

Lastly, we note that a special case of our work has also been considered in the literature by Viard. We derived our work independently of Viard.

\begin{remark} Consider a finite subset $S$ of $\mathbb{N}^2$. Next, for all $r = (i,j) \in S$, define $F_r = \{ (i_1,j) \in S : i_1 \geq i \} \cup \{ (i, j_1) : j_1 \geq j \}$, and define $\mathcal{F}' = \{F_r : r \in S \}$. This construction is related to the tools we used in Section 3 for the following reason. By using the same explanation as the one we gave for why $\{H_r : r \in \lambda / \mu\}$ is shellable, we observe that $\mathcal{F}'$ is shellable and that its unique transversal $t'$ is defined by $t'(F_r) = r$ for all $r \in S$. \\

Let $S$, $\mathcal{F}'$, and $t'$ be as described in the above paragraph. Viard \cite{ViardPaper, ViardThesis} considered objects that are equivalent to configurations of $t'$ and permutations that satisfy such configurations. Viard \cite{ViardPaper, ViardThesis} asserted that he has established one direction of a special case of Theorem \ref{good marriage} by claiming to have proved a statement equivalent to asserting that all configurations $f$ of $t'$ are satisfied by at least one permutation $\sigma : S \to S$. In particular, using his claim, he derives two consequences that implies the special case of Theorem \ref{theorem : partial product formula} for permutations $\sigma : S \to S$ and configurations $f$ of $t'$. There are two versions of his arguments (a less general version in \cite{ViardPaper} and a more general version in \cite{ViardThesis}), both versions are different from our proof of Theorem \ref{good marriage}.
\end{remark}


\section*{Acknowledgements}\label{sec:acknow} The author would like to thank Stephanie van~Willigenburg for her guidance and advice during the development of this paper. Furthermore, the author gratefully acknowledges an anonymous referee who identified a gap in the first half of the proof of Theorem 3.14  in an older draft of this paper.
\bibliographystyle{amsplain}

\end{document}